\documentclass{amsart}
\usepackage[utf8]{inputenc}
\usepackage[english]{babel}
\usepackage{cite}
\usepackage{amsthm,amsmath,amssymb,hyperref,amsfonts,tikz,adjustbox}
\usepackage{mathrsfs}
\usetikzlibrary{cd}
\numberwithin{equation}{section}
\usepackage[shortlabels]{enumitem}
\usepackage{imakeidx}

\usepackage{tikz}
\usetikzlibrary{cd}
\usetikzlibrary{calc} 
\usetikzlibrary{decorations.pathmorphing} 

\tikzset{curve/.style={settings={#1},to path={(\tikztostart)
    .. controls ($(\tikztostart)!\pv{pos}!(\tikztotarget)!\pv{height}!270:(\tikztotarget)$)
    and ($(\tikztostart)!1-\pv{pos}!(\tikztotarget)!\pv{height}!270:(\tikztotarget)$)
    .. (\tikztotarget)\tikztonodes}},
    settings/.code={\tikzset{quiver/.cd,#1}
        \def\pv##1{\pgfkeysvalueof{/tikz/quiver/##1}}},
    quiver/.cd,pos/.initial=0.35,height/.initial=0}

\tikzset{tail reversed/.code={\pgfsetarrowsstart{tikzcd to}}}
\tikzset{2tail/.code={\pgfsetarrowsstart{Implies[reversed]}}}
\tikzset{2tail reversed/.code={\pgfsetarrowsstart{Implies}}}
\tikzset{no body/.style={/tikz/dash pattern=on 0 off 1mm}}

\newtheorem{theorem}{Theorem}[section]
\newtheorem{corollary}[theorem]{Corollary}
\newtheorem{lemma}[theorem]{Lemma}
\newtheorem{proposition}[theorem]{Proposition}

\theoremstyle{definition}
\newtheorem{definition}[theorem]{Definition}

\theoremstyle{remark}
\newtheorem{remark}[theorem]{Remark}

\newcommand*{\im}{\operatorname{im}}

\newcommand*{\lmod}[1]{#1\text{-}\mathbf{Mod}}
\newcommand*{\rmod}[1]{\mathbf{Mod}\text{-}#1}
\renewcommand{\hom}{\operatorname{Hom}}

\title{Homology of Epsilon-Strongly Graded Algebras}



\author{Emmanuel Jerez}
\address{Guangdong Technion Israel Institute of Technology, Shantou, Guangdong
Province, China}
\email{ars.ejerez@icloud.com}

\begin{document}

\begin{abstract}
    Let $G$ be a group and $S$ a unital epsilon-strongly $G$-graded algebra.
    We construct spectral sequences converging to the Hochschild (co)homology of $S$.
    Each spectral sequence is expressed in terms of the partial group (co)homology of $G$ with coefficients in the Hochschild (co)homology of the degree-one component of $S$.
    Moreover, we show that the homology spectral sequence decomposes according to the conjugacy classes of $G$, and, by means of the globalization functor, its $E^2$-page can be identified with the ordinary group homology of suitable centralizers.
\end{abstract}

\maketitle

\section*{Introduction}

Epsilon-strongly graded algebras were introduced in \cite{Article_Nystedt-Oinert-Pinedo_2018_EGRSAS} as a natural generalization of unital strongly graded algebras.
In that same work,
it was shown that this class includes both strongly graded rings and unital partial crossed products and, a simple computation also shows that it contains twisted partial crossed products.
Moreover,
other important families of algebras are epsilon-strongly graded,
such as Morita rings \cite{Article_Nystedt-Oinert-Pinedo_2018_EGRSAS},
Leavitt path algebras \cite{Article_Nystedt-Oinert_2020_GGOLPA},
corner skew polynomial rings \cite{Article_Lannstrom_2020_TGSACR},
and graded algebras with the so‑called \textit{good‑gradations} \cite{Article_Lundstrom-Oinert-Orozco-Pinedo_2025}.
This unified framework therefore allows one to study the (co)homology of all these examples in a single setting.

In \cite{Article_Alvares-Alves-Redondo_2017_COPSP},
the authors introduce the partial group cohomology \(H^\bullet_{\mathrm{par}}(G,M)\) of a group \(G\)
with coefficients in a left $K_{\mathrm{par}}G$‑module $M$,
and apply it to compute a spectral sequence converging to the cohomology of a partial crossed product \(A\!\rtimes G\),
where \(A\) is a unital algebra on which \(G\) acts by a unital partial action \(\alpha=(A_g,\alpha_g)_{g\in G}\).
Later, \cite{Article_Dokuchaev-Jerez_2023_TTPGAAOPCP} extended this to the twisted setting, replacing $A\rtimes G$ by a twisted partial crossed product $A\rtimes_{\sigma}G$ determined by a unital twisted partial action $(A_g,\alpha_g)_{g \in G}$ and a twist $\sigma$.  In the homology, case they produce a first quadrant spectral sequence
\[
  E^{2}_{p,q} = H_{p}^\mathrm{par}\bigl(G,\,H_{q}(A,M)\bigr)\Longrightarrow H_{p+q}(A\rtimes_{\sigma}G,M),
\]
where $M$ is now an $(A\rtimes_{\sigma}G)$-bimodule and $H_{\bullet}^\mathrm{par}(G,-)$ denotes the partial group homology from \cite{Article_Alves-Dokuchaev-Kochloukova_2020_HACVTPGA}.  Although the construction in \cite{Article_Dokuchaev-Jerez_2023_TTPGAAOPCP} makes essential use of the twist $\sigma$, the twist itself does not explicitly appear in the resulting spectral sequence.

The main achievements of this paper is to present both spectral sequences within the unified framework of epsilon‑strongly graded algebras and to give a detailed description of the components of the homology spectral sequence,
thereby generalizing the results of \cite{Article_Lorenz_1992_OHGA} on strongly graded algebras.
We show that each spectral sequence depends only on the underlying epsilon‑strong grading of the algebra, and we derive general spectral sequences that compute the Hochschild (co)homology of any epsilon‑strongly graded algebra.

Section \ref{sec: preliminaries} reviews partial group representations and epsilon‑strong gradings.
In Section \ref{sec: homology of epsilon-strongly graded algebras}, Theorem~\ref{t: homology spectral sequence of an epsilon ring} and Theorem~\ref{t: cohomology spectral sequence of an epsilon ring} establish first‑ and second quadrant spectral sequences converging to the Hochschild homology and cohomology of an epsilon‑strongly \(G\)-graded algebra $S$.
These sequences are interpreted in terms of the partial group (co)homology of $G$ (see Definition~\ref{d: partial group cohomology}) with coefficients in the Hochschild (co)homology of degree-one component \(A:=S_1\) of \(S\).

In Section \ref{sec: splitting by conjugate classes} we show how the homology spectral sequence splits into summands indexed by conjugacy classes of \(G\), yielding the decomposition of Theorem~\ref{t: spectral sequence split}.
In the second part of this section we use the globalization functor \(\Lambda: \lmod{K_\mathrm{par}G} \to \lmod{KG}\)
(from \cite[Theorem 3.24]{PhDThesis_Jerez_2024}) to recover and extend Lorenz’s classical results \cite{Article_Lorenz_1992_OHGA} on the homology of strongly graded algebras.
Then Theorem~\ref{t: main} shows that,
on the second page of the homology spectral sequence,
each term \(H_p^{\mathrm{par}}(G,H_q(A,M_{\overline g}))\) is identified with the group homology of the centralizer $\mathcal C_g$ of \(g \in G\),
and—under mild hypotheses—the computation further reduces to the partial group homology of $\mathcal C_g$ itself.

\section{Preliminaries}
\label{sec: preliminaries}

From now on, we will assume that \(K\) is a commutative ring with identity and \(G\) is an abstract group.
We use the words \textit{module} and \textit{algebra} to refer to a \(K\)-module and a \(K\)-algebra with unit, respectively,
unless otherwise specified.

\subsection{Partial representations of groups}

We start by introducing the basic concepts of partial group representations used in this work, based on \cite{Book_Exel_2017_PDSFBAA}.

A \textbf{partial representation} of $G$ on the $K$-module $M$ is a map $ \pi : G \rightarrow \operatorname{End}_K(M)$ such that, for any $s,t \in G$, we have:
\begin{enumerate}[(a)]
    \itemsep0em 
    \item \(\pi_{s} \pi_{t} \pi_{t^{-1}} = \pi_{st} \pi_{t^{-1}}\),
    \item \(\pi_{s^{-1}} \pi_{s} \pi_{t} = \pi_{s^{-1}} \pi_{st}\),
    \item \(\pi_{1_G} = \operatorname{id}_{M}\).
\end{enumerate}
Let $\pi: G \rightarrow \operatorname{End}_K(M)$ and $\pi': G \rightarrow \operatorname{End}_K(W)$ be two partial representations of $G$.
A \textbf{morphism of partial representations} is a morphism of $K$-modules $f: M \rightarrow W$,
such that $f \circ \pi_{g} = \pi'_{g} \circ f$ $\forall g \in G$.

We write $\mathbf{ParRep}_{G}$ for the category of partial representations of $G$.  Let $\mathcal{S}(G)$ be Exel’s semigroup of $G$, i.e., the inverse monoid generated by the set of symbols $\{[t]\mid t\in G\}$ with relations
\[
    [1_G]=1,\qquad 
    [s^{-1}][s][t]=[s^{-1}][st],\qquad 
    [s][t][t^{-1}]=[st][t^{-1}].
\]
For $w\in\mathcal{S}(G)$, write $w^*$ for its inverse, and denote the set of idempotent elements of \(\mathcal{S}(G)\) by \(E(\mathcal{S}(G))\).
It is well known that \(E(\mathcal{S}(G))\) is the subsemigroup of \(\mathcal{S}(G)\) generated by the set \(\{ e_g = [g][g^{-1}] : g \in G \}\).
The \textbf{partial group algebra} $K_{\mathrm{par}}G$ of \(G\) is the semigroup $K$-algebra generated by \(\mathcal{S}(G)\),
and $\mathcal{B}\subset K_{\mathrm{par}}G$ its commutative subalgebra generated by $E(\mathcal{S}(G))$.
The generators $[g]$ and basic idempotents $e_g$ of \(K_\mathrm{par}G\) satisfy the following relations, which we will use throughout the paper:
\begin{lemma}[Proposition 9.8 \cite{Book_Exel_2017_PDSFBAA}] \label{p: computations rules}
    The following equalities hold in \(K_\mathrm{par}G\) for all \(g,h \in G\):
    \begin{enumerate}[(i)]
        \item $[g]e_h = e_{gh}[g]$,
        \item \([g]e_{h^{-1}}[g^{-1}] = e_{gh}e_{g}\),
        \item $e_g e_h = e_h e_g$.
    \end{enumerate}
\end{lemma}
The \(K\)-submodule \(\mathcal{B}\subseteq K_{\mathrm{par}}G\) carries left- and right-$K_{\mathrm{par}}G$–module structures (though not a bimodule) via
\[
    [g]\triangleright u=[g]\,u\,[g^{-1}],
    \quad
    u\triangleleft[g]=[g^{-1}]\,u\,[g],
\]
for all $g\in G$ and $u\in\mathcal{B}$.
One verifies that
\begin{enumerate}[(i)]
    \item $w \triangleright u = wu$ for all $w, u \in \mathcal{B}$,
    \item $u \triangleleft w = wu$ for all $w, u \in \mathcal{B}$,
    \item $[g] \triangleright 1 = e_{g}$ for all $g \in G$,
    \item $1 \triangleleft [g] = e_{g^{-1}}$ for all $g \in G$.
\end{enumerate}

Hence, we obtain a well-defined morphism of right $K_\mathrm{par}G$-modules
\begin{align} \label{eq: epsilon map}
    \varepsilon: K_\mathrm{par}G   &\to \mathcal{B} \\
                    z       &\mapsto  1 \triangleleft z. \notag
\end{align}


\begin{proposition}[Proposition 10.5 \cite{Book_Exel_2017_PDSFBAA}] \label{p: KparG universal property} 
    The map
    \[
       g \in G \mapsto [g] \in  K_\mathrm{par}G
    \]
    is a partial representation, which we will call the \textbf{universal partial representation} of $G$. In addition, for any partial representation $\pi$ of $G$ in a unital $K$-algebra $A$ there exists a unique algebra homomorphism $\phi : K_\mathrm{par}G \rightarrow A$, such that $\pi(g) = \phi([g])$, for any $g \in G.$
\end{proposition}

A direct consequence of Proposition~\ref{p: KparG universal property} is the following well-known theorem:

\begin{theorem} \label{t: partial representations and KparG-modules isomorphism}
    The categories $\textbf{ParRep}_G$ and $K_\mathrm{par}G$-\textbf{Mod} are isomorphic.
\end{theorem}

The key constructions in Theorem~\ref{t: partial representations and KparG-modules isomorphism} are as follows: if $\pi: G \to \operatorname{End}_K(M)$ is a partial group representation, then $M$ becomes a $K_\mathrm{par}G$-module with the action defined as $[g] \cdot m := \pi_g(m)$. Conversely, if $M$ is a $K_\mathrm{par}G$-module, we obtain a partial representation $\pi: G \to \operatorname{End}_K(M)$ such that $\pi_g(m):= [g] \cdot m$. Keeping in mind Theorem~\ref{t: partial representations and KparG-modules isomorphism}, we use the terms \textit{partial group representation} and $K_\mathrm{par}G$\textit{-module} interchangeably.

To conclude the review of partial group representations,
we recall the definition of partial group (co)homology,
defined using the partial group algebra $K_{\mathrm{par}}G$ and its commutative subalgebra $\mathcal{B}$.
\begin{definition}[\hspace{-0.005em}\cite{Article_Alvares-Alves-Redondo_2017_COPSP}, \cite{Article_Alves-Dokuchaev-Kochloukova_2020_HACVTPGA}]
    \label{d: partial group cohomology}
    Let $G$ be a group, we define the partial group homology of $G$ with coefficients in a left $K_\mathrm{par}G$-module $M$ as
    \[
        H_{\bullet}^\mathrm{par}(G,M) := \operatorname{Tor}_{\bullet}^{K_\mathrm{par}G}(\mathcal{B},M).
    \]
    Analogously, we define the partial group cohomology of \(G\) with coefficients in \(M\) as
    \[
        H^\bullet_\mathrm{par}(G,M) := \operatorname{Ext}^\bullet_{K_\mathrm{par}G}(\mathcal{B}, M).
    \]
\end{definition}

\subsection{Epsilon-strongly graded algebras}
Following \cite{Article_Nystedt-Oinert-Pinedo_2018_EGRSAS},
a unital \(G\)-graded algebra \(S\) is \emph{epsilon-strongly graded} if:
    \begin{enumerate}[(i)]
        \item \(S_{g}S_{g^{-1}}\) is a unital ideal of \(S_1\) for all \(g \in G\),
        \item \(S_{g^{-1}}S_{g}S_{h}  = S_{g^{-1}}S_{gh}\) for all \(g,h \in G\),
        \item \(S_{h}S_{g}S_{g^{-1}} = S_{hg} S_{g^{-1}}\) for all \(g,h \in G\).
    \end{enumerate}
We denote by $1_{g}$ the unit of $S_{g}S_{g^{-1}}$, and by $s_g$ an arbitrary homogeneous element of degree $g$ in $S$ (i.e., $s_g \in S_g$).
For all \(g \in G\) there exists a finite family \(\{ (L_{g}^{i}, R_{g^{-1}}^{i}) \in S_{g} \times S_{g^{-1}}\}\) such that
\begin{equation}
    1_{g} = \sum_{i} L_{g}^{i} R_{g^{-1}}^{i}.
    \label{eq: 1g sum of LR}
\end{equation}
Simple computations show that

\vspace*{-10pt}
\begin{minipage}[t]{0.45\textwidth}
    \begin{equation}
        1_{g}s_{g} = s_{g} = s_{g}1_{g^{-1}} \label{eq: epsilon-strongly graded algebras properties 1}
    \end{equation}
\end{minipage}
\begin{minipage}[t]{0.45\textwidth}
    \begin{equation}
        s_{g}1_{h} = 1_{gh}s_{g}, \label{eq: epsilon-strongly graded algebras properties 2}
    \end{equation}
\end{minipage}

\medskip
\noindent for all \(g,h \in G\) and \(s_{g} \in S_g\).

\section{Homology of epsilon-strongly graded algebras}
\label{sec: homology of epsilon-strongly graded algebras}

Let $S$ be a unital, epsilon‑strongly $G$‑graded algebra over a commutative ring $K$.
Write
\[
  A \;=\; S_{1}, 
  \quad
  A_{g} \;=\; S_{g}\,S_{g^{-1}}
  \quad (g\in G).
\]
Henceforth we shall assume that $A$ is projective as a $K$‑module.

\subsection{Homology}
\label{sub:homology}

Our objective is to describe the homology of \(S\) in terms of the homology of \(A\),
with this in mind, we define the functor
\begin{equation}
    \mathcal{F} := A \otimes_{A^{e}} - : \lmod{S^{e}} \to \lmod{K}
    \label{eq: F1 functor}
\end{equation}
Note that the functor \(\mathcal{F}\) differs from \(H_{0}(A, -)\) in its source category.  
Hence, it is possible that \(\operatorname{L}_{\bullet} \mathcal{F} (M) \ncong H_{\bullet}(A, M)\) may not hold. 

\begin{proposition} \label{p: F HA isomorphism}
    Let \(S\) be a epsilon-strongly \(G\)-graded algebra.
    Then,
    \begin{enumerate}[(i)]
        \item \(S\) is projective as left and right \(A\)-module,
        \item \(S^{e}\) is projective as \(A^{e}\)-module,
        \item \(\mathbf{L}_{\bullet}\mathcal{F}(M) \cong H_{\bullet}(A, M)\) for every \(S^{e}\)-module \(M\).
    \end{enumerate}
\end{proposition}
\begin{proof}
    Since \(A_{g}\) is a unital ideal of \(A\), then
    \[
        A = A_{g} \oplus (1-1_{g})A,
    \]
    whence we conclude that each \(A_{g}\) is projective as left and right \(A\)-module.
    By \cite[Proposition 7]{Article_Nystedt-Oinert-Pinedo_2018_EGRSAS}, \(S_{g}\) is projective as both a left \(A_{g}\)-module and right \(A_{g^{-1}}\)-module.
    Moreover, by \cite[Proposition 1.4]{Book_De-Ingraham_2006_SAOCR}, each \(S_{g}\) is projective as both a left and right \(A\)-module.
    Hence, \(S\) is projective as both a left and right \(A\)-module.
    This concludes the proof of (i).
    The proof of (ii) follows from a straightforward dual basis argument.
    Finally, by (ii) and \cite[Proposition 1.4]{Book_De-Ingraham_2006_SAOCR}, we know that any projective resolution of \(M\) as an \(S^{e}\)-module is also a projective resolution of \(M\) as an \(A^{e}\)-module, which implies (iii).
\end{proof}

Let \(M\) be an \(S^{e}\)-module.
Although \(M\) is not necessarily a \(K_\mathrm{par}G\)-module, we can define a partial representation of \(G\) on \(A \otimes_{A^{e}} M\).
For each $g\in G$, choose a family \(\{(L^i_g,R^i_{g^{-1}})\}_i\subset S_g\times S_{g^{-1}}\) satisfying \eqref{eq: 1g sum of LR},
and define \(\pi_{g}:G \to \operatorname{End}_{K}(M)\) such that
\begin{equation}\label{eq: partial representation on H0}
    \pi_{g}(a \otimes_{A^{e}} m) := \sum_{i} 1 \otimes_{A^{e}}R_{g}^{i} a m L_{g^{-1}}^{i}.
\end{equation}
To show that \(\pi_{g}\) is well‑defined on the tensor product over \(A^{e}\), one computes:
\begin{align*}
    \sum_{i} 1 \otimes_{A^{e}}R_{g}^{i} a m L_{g^{-1}}^{i} 
    &= \sum_{i} 1 \otimes_{A^{e}}R_{g}^{i} a 1_{g^{-1}} m L_{g^{-1}}^{i}
    = \sum_{i,j} 1 \otimes_{A^{e}}R_{g}^{i} a L_{g^{-1}}^{j} R_{g}^{j} m L_{g^{-1}}^{i} \\
    &= \sum_{i,j} 1 \otimes_{A^{e}}R_{g}^{j} m L_{g^{-1}}^{i} R_{g}^{i} a L_{g^{-1}}^{j}
    = \sum_{j} 1 \otimes_{A^{e}}R_{g}^{j} m 1_{g^{-1}} a L_{g^{-1}}^{j} \\
    &= \sum_{j} 1 \otimes_{A^{e}}R_{g}^{j} m a L_{g^{-1}}^{j}.
\end{align*}
Furthermore, analogously to the above computations we show that the definition of \(\pi_{g}\) does not depend on the choice of the \((L_{g}^{i}, R_{g^{-1}}^{i})\).
Indeed, if \(\{ (U_{g}^{i}, W_{g^{-1}}^{i}) \in S_{g} \times S_{g^{-1}}\}\) also satisfies \eqref{eq: 1g sum of LR}, then we have
\begin{align*}
    \sum_{i} 1 \otimes_{A^{e}}R_{g}^{i} m L_{g^{-1}}^{i} 
    &= \sum_{i} 1 \otimes_{A^{e}}R_{g}^{i} 1_{g^{-1}} m L_{g^{-1}}^{i}
    = \sum_{i,j} 1 \otimes_{A^{e}}R_{g}^{i} U_{g^{-1}}^{j} W_{g}^{j} m L_{g^{-1}}^{i} \\
    &= \sum_{i,j} 1 \otimes_{A^{e}}W_{g}^{j} m L_{g^{-1}}^{i} R_{g}^{i} U_{g^{-1}}^{j}
    = \sum_{j} 1 \otimes_{A^{e}}W_{g}^{j} m 1_{g^{-1}} U_{g^{-1}}^{j} \\
    &= \sum_{j} 1 \otimes_{A^{e}}W_{g}^{j} m U_{g^{-1}}^{j}.
\end{align*}

\begin{proposition} \label{p: partial representation on coinvariants}
    Let \(M\) be an \(S^{e}\)-module.
    Then, the map \(\pi: G \to \operatorname{End}_{K}(A \otimes_{A^{e}} M)\) defined by \eqref{eq: partial representation on H0} is a partial representation.
\end{proposition}

\begin{proof}
    It is clear that \(\pi_{1} = \operatorname{id}_{M}\).
    Now observe that if \(g \in G\) and \(m \in M\),
    then
    \begin{equation}\label{eq: idempotent action on A otimes M}
        \pi_{g} \pi_{g^{-1}}(1 \otimes_{A^{e}}m) = 1_{g} \otimes_{A^{e}} m.
    \end{equation}
    Indeed,
   \begin{align*}
       \pi_{g} \pi_{g^{-1}}(1 \otimes_{A^{e}} m)
       &= \sum_{i,j} 1 \otimes_{A^{e}}R_{g}^{i} R_{g^{-1}}^{j}m L_{g}^{j}L_{g^{-1}}^{i}
       = \sum_{i,j} 1 \otimes_{A^{e}} m L_{g}^{j}L_{g^{-1}}^{i}R_{g}^{i} R_{g^{-1}}^{j} \\
       &= \sum_{j} 1 \otimes_{A^{e}} m L_{g}^{j} 1_{g^{-1}} R_{g^{-1}}^{j}
       = \sum_{j} 1 \otimes_{A^{e}} m L_{g}^{j} R_{g^{-1}}^{j}\\
       &= 1_g \otimes_{A^{e}} m.
   \end{align*}
   Therefore, by \eqref{eq: idempotent action on A otimes M}, we obtain
   \[
       \pi_{g^{-1}} \pi_{g} \pi_{h}(1 \otimes_{A^{e}} m) = \sum_{i} 1_{g^{-1}} \otimes_{A^{e}} R_{h}^{i} m L_{h^{-1}}^{i}.
   \]
   On the other hand,
   \begin{align*}
       &\pi_{g^{-1}} \pi_{gh}(1 \otimes_{A^{e}} m) \\
       &= \sum_{i,j} 1 \otimes_{A^{e}} R_{g^{-1}}^{j} R_{gh}^{i} m L_{h^{-1}g^{-1}}^{i} L_{g}^{j}
       \overset{\eqref{eq: epsilon-strongly graded algebras properties 1}}{=} \sum_{i,j} 1 \otimes_{A^{e}} R_{g^{-1}}^{j} R_{gh}^{i} 1_{h^{-1}} m L_{h^{-1}g^{-1}}^{i} L_{g}^{j} \\
       & \overset{\eqref{eq: 1g sum of LR}}{=} \sum_{i,j,n} 1 \otimes_{A^{e}} R_{g^{-1}}^{j} R_{gh}^{i} L_{h^{-1}}^{n}R_{h}^{n} m L_{h^{-1}g^{-1}}^{i} L_{g}^{j} 
       = \sum_{i,j,n} 1 \otimes_{A^{e}} R_{h}^{n} m L_{h^{-1}g^{-1}}^{i} L_{g}^{j} (R_{g^{-1}}^{j} R_{gh}^{i} L_{h^{-1}}^{n}) \\
       &= \sum_{i,n} 1 \otimes_{A^{e}} R_{h}^{n} m L_{h^{-1}g^{-1}}^{i} 1_{g} R_{gh}^{i} L_{h^{-1}}^{n}
       \overset{\eqref{eq: epsilon-strongly graded algebras properties 1}}{=} \sum_{i,n} 1 \otimes_{A^{e}} R_{h}^{n} m L_{h^{-1}g^{-1}}^{i} R_{gh}^{i} L_{h^{-1}}^{n} \\
       &= \sum_{n} 1 \otimes_{A^{e}} R_{h}^{n} m 1_{h^{-1}g^{-1}} L_{h^{-1}}^{n}
      \overset{\eqref{eq: epsilon-strongly graded algebras properties 2}}{=} \sum_{n} 1_{g^{-1}} \otimes_{A^{e}} R_{h}^{n} m L_{h^{-1}}^{n}.
   \end{align*}
   Analogously, one verifies that \(\pi_{g} \pi_{h} \pi_{h^{-1}} = \pi_{gh}\pi_{h^{-1}}\).
\end{proof}

\begin{remark}
    Note that
    the assignment \(e_{g} \in E(\mathcal{S}(G)) \mapsto 1_{g} \in S\) 
    extends uniquely to an algebra homomorphism $\mathcal{B}\to S$, endowing $S$ with a $\mathcal{B}$‑algebra structure such \(e_{g} \cdot s = 1_{g}s\) for all \(s \in S\) and \(g \in G\).
\end{remark}

\begin{lemma}
    \label{l: S bimodule structure of X otimesB S}
    Let \(X\) be a right \(K_\mathrm{par}G\)-module and \(M\) be an \(S^{e}\)-module.
    Then, \(X \otimes_{\mathcal{B}} M\) is a \(S\)-bimodule with the following structure:
    \[
        s_{g} \cdot (x \otimes_{\mathcal{B}} m) := x [g^{-1}] \otimes_{\mathcal{B}} s_{g} m \quad \text{and} \quad (x \otimes_{\mathcal{B}} m) \cdot s_{g} := x \otimes_{\mathcal{B}} m s_{g},
    \]
    where \(s_g \in S_g\), \(x \in X\), \(m \in M\), and \(g \in G\).
\end{lemma}
\begin{proof}
    It is clear that the right action is well-defined.
    Note that if the left action is well-defined, the bimodule structure will also be well-defined.
    Let \(x \in X\), \(m \in M\), \(g, h \in G\), and \(s_{g}, s_{h} \in S\).
    Then,
    \begin{align*}
        s_{h} \cdot (s_{g} \cdot (x \otimes_{\mathcal{B}} m)) 
        &= x [g^{-1}][h^{-1}] \otimes_{\mathcal{B}} s_{h}s_{g} m
        = x [g^{-1}h^{-1}] e_{h} \otimes_{\mathcal{B}} s_{h}s_{g} m \\
        &= x [g^{-1}h^{-1}] \otimes_{\mathcal{B}} 1_{h} s_{h}s_{g} m
        = x [g^{-1}h^{-1}] \otimes_{\mathcal{B}} s_{h}s_{g} m \\
        &= s_{h} s_{g} \cdot (x \otimes_{\mathcal{B}} m).
    \end{align*}
\end{proof}

\begin{proposition} \label{p: main functor isomorphism}
    The functors
    \[
        -\otimes_{K_\mathrm{par}G} (A \otimes_{A^e} - ): \rmod K_\mathrm{par}G \times \lmod{S^{e}}  \to \lmod{K}
    \]
    and
    \[
        (- \otimes_{\mathcal{B}} S) \otimes_{S^e} - : \rmod K_\mathrm{par}G \times \lmod{S^{e}} \to \lmod{K} 
    \] 
    are naturally isomorphic.
\end{proposition}

\begin{proof}
    Let $X$ be a right $K_\mathrm{par}G$-module, and \(M\) be an \(S^{e}\)-module.
    Define
    \begin{align*}
        \gamma : X \otimes_{K_\mathrm{par}G} (A\otimes_{A^e}M) & \to (X \otimes_{\mathcal{B}} S)  \otimes_{S^e}M \\ 
        x \otimes_{K_\mathrm{par}G}(a\otimes_{A^e}m) &\mapsto (x \otimes_{\mathcal{B}} 1_S) \otimes_{S^e} a m.
    \end{align*}
    Note that \(\gamma\) is well-defined since
    \begin{align*}
        \gamma(x[g] \otimes_{K_\mathrm{par}G} (1 \otimes_{A^{e}} am))
        &=(x [g] \otimes_{\mathcal{B}} 1_S) \otimes_{S^e} a m 
        = (x [g]e_{g^{-1}} \otimes_{\mathcal{B}} a) \otimes_{S^e} m \\
        &= (x [g] \otimes_{\mathcal{B}} a1_{g^{-1}}) \otimes_{S^e} m 
        = \sum_{i} (x [g] \otimes_{\mathcal{B}} aL_{g^{-1}}^{i} R_{g}^{i}) \otimes_{S^e} m \\
        &= \sum_{i} aL_{g^{-1}}^{i} \cdot (x \otimes_{\mathcal{B}} R_{g}^{i}) \otimes_{S^e} m
        = \sum_{i} (x \otimes_{\mathcal{B}} 1_{S}) \otimes_{S^e} R_{g}^{i} m a L_{g^{-1}}^{i} \\
        &= \sum_{i}\gamma(x \otimes_{K_\mathrm{par}G} (1 \otimes_{A^{e}} R_{g}^{i}maL_{g^{-1}}^{i}))
        = \gamma(x \otimes_{K_\mathrm{par}G} [g](1 \otimes_{A^{e}} ma)).
    \end{align*}
    Furthermore, the inverse of \(\gamma\) is given by
    \[
        \gamma^{-1}\big((x \otimes_{\mathcal{B}} s) \otimes_{S^e} m \big) := x \otimes_{K_\mathrm{par}G} (1 \otimes_{A^{e}} s m).
    \]
    Observe that for all \(g,h \in G\), \(s_{g} \in S_{g}\) and \(s_{h} \in S_{h}\) we have that
    \begin{align*}
        \gamma^{-1}(s_{g}(x \otimes_{\mathcal{B}} s) s_{h} \otimes_{S^e} m)
        &= \gamma^{-1}((x[g^{-1}] \otimes_{\mathcal{B}} s_{g} s s_{h}) \otimes_{S^e} m) \\
        &= x[g^{-1}] \otimes_{K_\mathrm{par}G} (1 \otimes_{A^{e}} s_{g} s s_{h} m) \\
        &= \sum_{i} x \otimes_{K_\mathrm{par}G} (1 \otimes_{A^{e}} L_{g^{-1}}^{i}s_{g} s s_{h} m R_{g}^{i}) \\
        &= \sum_{i} x \otimes_{K_\mathrm{par}G} (1 \otimes_{A^{e}} s s_{h} m R_{g}^{i}L_{g^{-1}}^{i}s_{g} ) \\
        &= x \otimes_{K_\mathrm{par}G} (1 \otimes_{A^{e}} s s_{h} m 1_{g}s_{g} ) \\
        &= \gamma^{-1}\big((x \otimes_{\mathcal{B}} s) \otimes_{S^e} s_{h} m s_{g}\big).
    \end{align*}
    and
    \begin{align*}
        \gamma^{-1}((x \cdot e_{h} \otimes_{\mathcal{B}} s) \otimes_{S^e} m)
        &= x \cdot e_{h} \otimes_{K_\mathrm{par}G} (1 \otimes_{A^{e}} s m)
        \overset{\eqref{eq: idempotent action on A otimes M}}{=} x \otimes_{K_\mathrm{par}G} (1 \otimes_{A^{e}} (e_{h} \cdot s) m) \\
        &= \gamma^{-1}((x \otimes_{\mathcal{B}} (e_{h} \cdot s)) \otimes_{S^e} m).
    \end{align*}
    Hence, \(\gamma^{-1}\) is well-defined, and consequently, \(\gamma\) is an isomorphism.
\end{proof}

\begin{corollary} \label{c: H0 F isomorphism}
    There is an isomorphism of functors:
    \[
        H_{0}^\mathrm{par}(G, -) \circ \mathcal{F} \cong H_{0}(S, -).
    \]
\end{corollary}
\begin{proof}
    By Proposition~\ref{p: main functor isomorphism} we have an natural isomorphis of functors
    \[
        \mathcal{B}\otimes_{K_\mathrm{par}G} (A \otimes_{A^e} - ) \cong (\mathcal{B} \otimes_{\mathcal{B}} S) \otimes_{S^e} -.
    \]
    Moreover, note that \(\mathcal{B} \otimes_{\mathcal{B}} S\) is isomorphic to \(S\) as \(S\)-bimodule.
    Whence we obtain the desired isomorphism.
\end{proof}

\begin{lemma}\label{l: acylic heart}
    Define the functor \(\mathfrak{T}: \rmod{K_\mathrm{par}G} \to \lmod{S^{e}}\)
    such that \(\mathfrak{T}(X) := X \otimes_{\mathcal{B}} S\).
    Then, \(\mathbf{L}_{\bullet}(\mathfrak{T}) \cong \operatorname{Tor}_{\bullet}^{\mathcal{B}}(-, S)\).
    In particular, \(\mathbf{L}_{n}(\mathfrak{T})(\mathcal{B}) = 0\) for all \(n \geq 1\).
\end{lemma}
\begin{proof}
    By \cite[Lemma 3.22]{Article_Dokuchaev-Jerez_2023_OPSP},
    in the particular case $\mathcal{H}_{\mathrm{par}} = K_{\mathrm{par}}G$,
    it follows that any projective resolution in $\mathbf{Mod}\text{-}K_{\mathrm{par}}G$ is also a projective resolution in $\mathbf{Mod}\text{-}\mathcal{B}$.
    Hence the desired result follows.
\end{proof}

\begin{proposition} \label{p: F send projectives to H0 acyclics}
    \(\mathcal{F}\) sends projective $S$-bimodules to left $H_{0}^{\mathrm{par}}(G, -)$-acyclic modules.
\end{proposition}
\begin{proof}
    Let \(P\) be any projective $S^e$-module.
    Observe that
    \begin{align*}
        H^{\mathrm{par}}_{n}(G, A \otimes_{A^e}P)
        &\cong \mathbf{L}_n(- \otimes_{K_\mathrm{par}G} (A \otimes_{A^e}P))(\mathcal{B}) \\
        (\text{by Proposition \ref{p: main functor isomorphism}}) &\cong \mathbf{L}_n((- \otimes_{\mathcal{B}} S) \otimes_{S^e}P)(\mathcal{B}).
    \end{align*}
    Let $Q_{\bullet} \to \mathcal{B}$ a projective resolution of $\mathcal{B}$ in \textbf{Mod}-\(K_\mathrm{par}G\). Then,
    \[
        H_{n}^{\mathrm{par}}(G, A \otimes_{A^e}P)
        \cong
        H_n \left(  (Q_{\bullet}\otimes_{\mathcal{B}} S) \otimes_{S^e}P \right).
    \]
    By Lemma~\ref{l: acylic heart} the complex $Q_{\bullet}\otimes_{\mathcal{B}} S$ is exact for all $n \geq 1$.
    Then, the complex $(Q_{\bullet}\otimes_{\mathcal{B}} S) \otimes_{S^e}P$ is exact for all $n \geq 1$ since $P$ is projective as $S^e$-module, and so
    \[
        H^{\mathrm{par}}_{n}(G, A \otimes_{A^{e}}P) \cong  H_n \left(  (Q_{\bullet}\otimes_{\mathcal{B}} S) \otimes_{S^e}P \right) = 0, \forall n \geq 1,
    \]
    as we wanted to show.
\end{proof}

\begin{theorem} \label{t: homology spectral sequence of an epsilon ring}
    Let \(S\) be an epsilon-strongly \(G\)-graded algebra, let \(A := S_1\), and let \(M\) be an \(S^{e}\)-module. Then, there exists a homology spectral sequence that converges to the Hochschild homology of \(S\) with coefficients in \(M\): 
    \[
        E^{2}_{p,q} = H_{p}^\mathrm{par}(G, H_{q}(A, M)) \Rightarrow H_{p+q}(S, M).
    \]
\end{theorem}
\begin{proof}
    We aim to apply the Grothendieck spectral sequence Theorem (see \cite[Theorem 10.48]{Book_Rotman_2008_AITHA}).
    By Corollary~\ref{c: H0 F isomorphism} and Proposition~\ref{p: F send projectives to H0 acyclics} we know that \(H_{0}^\mathrm{par}(G, -) \circ \mathcal{F} \cong H_{0}(S, -)\)
    and that \(\mathcal{F}\) sends projective \(S^{e}\)-modules to left \(H_{0}^{\mathrm{par}}(G, -)\)-acyclic modules, respectively.
    Moreover, by (iii) of Proposition~\ref{p: F HA isomorphism} we know that \(\operatorname{L}_{\bullet}(\mathcal{F})(M) \cong H_{\bullet}(A, M)\) for every \(S^{e}\)-module \(M\).
\end{proof}

\subsection{Cohomology}
We now construct a spectral sequence converging to the Hochschild cohomology of an epsilon-strongly \(G\)-graded algebra \(S\) with coefficients in an \(S\)-bimodule \(M\).
For each \(g\), define \(\tau_{g}: \hom_{A^{e}}(A,M) \to \hom_{A^{e}}(A,M)\) by 
\[
    \tau_{g}(f)(a) := \sum_{i} aL_{g}^{i}f(1)R_{g^{-1}}^{i}.
\]
Dual to Proposition~\ref{p: partial representation on coinvariants}, one shows that \(\tau\) is a partial representation.
\begin{remark}
    Note that if $M = S$, then $\hom_{A^{e}}(A,S) \cong Z(A) \oplus \bigoplus_{g \neq 1} \hom_{A^{e}}(A, S_{g})$.
    Thus, the restriction of \(\tau\) to \(Z(A)\) recovers the partial representation of \(G\) on $Z(A)$ defined in
    \cite[Section 3]{Article_Nystedt-Oinert-Pinedo_2018_EGRSAS}.
\end{remark}

Consider the functor 
\[
    \mathcal{G} := \hom_{A^{e}}(A, -) : \lmod{S^{e}} \to \lmod{K}.
\]

By an argument analogous to that of Proposition~\ref{p: main functor isomorphism},
it follows that for any left \(K_\mathrm{par}G\)-module \(X\) and any \(S^{e}\)-module \(M\),
the map
\[
   \zeta: \hom_{K_\mathrm{par}G}(X, \hom_{A^{e}}(A, M)) \to  \hom_{S^{e}}(X \otimes_{\mathcal{B}} S, M)
\]
such that
\[
    \zeta(f)(x \otimes_{\mathcal{B}} s) := f_x\big(1_A\big) \cdot s,
\]
where \(f_x:= f(x) \in \hom_{A^{e}}(A,M)\),
is a well-defined isomorphism with inverse
\[
    \big(\zeta^{-1}(f')\big)_x(a) := f'(x \otimes_{\mathcal{B}} a).
\]
Furthermore, \(\zeta\) is natural in both variables \(X\) and \(M\).
Hence, we obtain the following isomorphism of functors:
\begin{equation}\label{eq: strange cohomology adjuntion 1}
    \hom_{K_\mathrm{par}G}(-, \hom_{A^{e}}(A, M)) \cong  \hom_{S^{e}}(- \otimes_{\mathcal{B}} S, M)
\end{equation}
and
\begin{equation}\label{eq: strange cohomology adjuntion 2}
    \hom_{K_\mathrm{par}G}(X, \hom_{A^{e}}(A, -)) \cong  \hom_{S^{e}}(X \otimes_{\mathcal{B}} S, -),
\end{equation}
for all \(K_\mathrm{par}G\)-modules \(X\) and all \(S^{e}\)-modules \(M\).
Consequently, if we put \(X = \mathcal{B}\) in \eqref{eq: strange cohomology adjuntion 2}, we obtain the isomorphism
\begin{equation}\label{eq: cohomology main functor isomorphism}
    H^{0}_{\mathrm{par}}(G, -) \circ \mathcal{G} \cong H^{0}(S, -).
\end{equation}

\begin{proposition}\label{p: mG send injectives to acyclics}
    The functor \(\mathcal{G}\) sends injective \(S^{e}\)-modules to left \(H^{0}_{\mathrm{par}}(G, -)\)-acyclic modules.
\end{proposition}
\begin{proof}
    Let \(Q\) be any injective $S^e$-module.
    Then, by Lemma~\ref{l: acylic heart}, $Q_{\bullet}\otimes_{\mathcal{B}} S$ is exact for all $n \geq 1$.
    Therefore,
    \begin{align*}
        H_{\mathrm{par}}^{n}(G, \hom_{A^{e}}(A, Q))
        &= \operatorname{Ext}_{K_\mathrm{par}G}^{n}(\mathcal{B}, \hom_{A^{e}}(A, Q)) \\
        &\cong
        H^n\big(\hom_{K_\mathrm{par}G}\big(P_\bullet, \hom_{A^{e}}(A, Q)\big)\big) \\
        (\text{by } \eqref{eq: strange cohomology adjuntion 1})
        &\cong
        H^n(\hom_{A^{e}}((P_\bullet \otimes_{\mathcal{B}} S), Q)) = 0,
    \end{align*}
    for all \(n \geq 1\) since \(Q\) is injective as an \(S^{e}\)-module.
    That is, \(\mathcal{G}(Q)\) is \(H^{0}_{\mathrm{par}}(G, -)\)-acyclic.
\end{proof}

Hence, combining isomorphism \eqref{eq: cohomology main functor isomorphism} and Proposition~\ref{p: mG send injectives to acyclics}, we obtain the cohomology version of Theorem~\ref{t: homology spectral sequence of an epsilon ring}.

\begin{theorem} \label{t: cohomology spectral sequence of an epsilon ring}
    Let \(S\) be an epsilon-strongly \(G\)-graded algebra, let \(A := S_1\), and let \(M\) be an \(S^{e}\)-module. Then, there exists a cohomology spectral sequence that converges to the Hochschild cohomology of \(S\) with coefficients in \(M\): 
    \[
        E_{2}^{p,q} = H^{p}_\mathrm{par}(G, H^{q}(A, M)) \Rightarrow H^{p+q}(S, M).
    \]
\end{theorem}

Observe that Theorems~\ref{t: homology spectral sequence of an epsilon ring} and \ref{t: cohomology spectral sequence of an epsilon ring} generalize the spectral sequences obtained in \cite{Article_Alvares-Alves-Redondo_2017_COPSP} and \cite{Article_Dokuchaev-Jerez_2023_TTPGAAOPCP}.

\section{Splitting by conjugate classes}
\label{sec: splitting by conjugate classes}

Our goal is to express the homology spectral sequence of Theorem~\ref{t: homology spectral sequence of an epsilon ring} as a direct sum of spectral sequences indexed by the conjugacy classes of \(G\),
thus generalizing the results of \cite{Article_Lorenz_1992_OHGA}.

Let $R$ be a $G$-graded $K$-algebra.
A $G$-graded $R$-bimodule $X$ is an $R$-bimodule with a $G$-grading such that
\(R_g \cdot X_h \subseteq X_{gh} \text{ and } X_h \cdot R_g \subseteq X_{hg}\).
A morphism $f:X \to Y$ of $G$-graded $R$-bimodules is a map of $R$-bimodules that respects the $G$-grading, i.e., \(f(X_g) \subseteq Y_g, \, \forall g \in G\).
We denote by \(\lmod{S^{e}}_{\mathrm{gr}}\) the category of $G$-graded $S$-bimodules.

\textbf{From now on} we fix an epsilon-strongly graded algebra $S$ and \(G\)-graded \(S^{e}\)-bimodule \(M\).
For $g \in G$ it will be convenient for us to denote the conjugacy class of $g$ by $\overline{g}$, and we write $\operatorname{Conj}(G)$ for the set of all conjugacy classes of $G$.
Moreover, we denote by $m_g$ a homogeneous element of degree $g$ in $M$ (i.e.\ $m_g\in M_g$).

It is well-known that the Hochschild complex \(C_{\bullet}(S, M)\) splits by the conjugacy classes of $G$ in the following way
\[
    C_n(S, M):=\bigoplus_{\overline{g} \in \operatorname{Conj}(G) } C_n^{\overline{g} }(S, M),
\]
where $C_n^{\overline{g} }(S, M)$ is the submodule of $C_n(S, M)$ spanned by
\[
    \{m_{g_0} \otimes s_{g_1} \otimes \ldots \otimes s_{g_n} \in C_n(S, M) : g_0g_1 \ldots g_n \in \overline{g}\}.
\]
Writing \(H_\bullet^{\overline{g}}(S, M):= H_\bullet(C_n^{\overline{g} }(S, M),b)\).
One obtains the direct sum decomposition
\begin{equation} \label{eq: Hn splitting by conjugacy classes}
    H_n(S, M):=\bigoplus_{\overline{g} \in \operatorname{Conj}(G) } H_n^{\overline{g} }(S, M).
\end{equation}

\begin{lemma} \label{l: EPgrad}
    For any $G$-graded $S$-bimodule $X$, there exists a projective resolution $(P_\bullet,\partial)$ of $X$ in $\lmod{S^e}$ such that $(P_\bullet,\partial)$ is also a projective resolution of $X$ in $\lmod{S^{e}}_\mathrm{gr}$.
    In particular, $\lmod{S^{e}}_{\mathrm{gr}}$ has enough projectives.
\end{lemma}

\begin{proof}
    It is enough to show that for any $G$-graded $S$-bimodule $X$ there exists a free $G$-graded $S$-bimodule $\tilde{X}$ (that is, $\tilde{X}$ is a free $S^e$-$module$), and a surjective map $f:\tilde{X} \to X$ of $G$-graded $S$-bimodules.
    Let $\tilde{X}$ be the free $S$-bimodule generated by the set $\sqcup_{g \in G} X_{g}$, thus $\tilde{X}$ has a natural $G$-grading given by 
    \[
        \tilde{X}_g := \bigoplus_{\substack{u,t,h \in G, \\uth=g}} S_u X_t S_h.
    \]
    Notice that the epimorphism of $f:\tilde{X} \to X$, obtained by the universal property of free modules from the inclusion map $\sqcup_{g \in G} X_{g} \to X$, preserves the $G$-grading.
\end{proof}

Now we consider the restriction
$$
    \mathcal{F}|_{\lmod{S^{e}}_{\mathrm{gr}}} \colon \lmod{S^e}_{\mathrm{gr}} \to \lmod{K_\mathrm{par}G}
$$
of $\mathcal{F}$ to $\lmod{S^e}_{\mathrm{gr}}$.
Since $\lmod{S^e}_{\mathrm{gr}}$ has enough projectives, the left derived functor of $\mathcal{F}|_{\lmod{S^e}_{\mathrm{gr}}}$ is well defined.
In particular, by Lemma~\ref{l: EPgrad}, for any $G$‑graded $S$‑bimodule $M$ we conclude that 
\[
    \mathbf{L}_\bullet(\mathcal{F}|_{\lmod{S^e}_{\mathrm{gr}}}(-))(X) = \mathbf{L}_\bullet(\mathcal{F}(-))(X) = H_\bullet(S, X).
\]

For a $G$‑graded $S$‑bimodule~$X$, set
\[
  X_{\overline g}
  :=
  \bigoplus_{h\in\overline g} X_h
  \quad(\overline g\in \operatorname{Conj}(G)).
\]
Since each $X_{\overline g}$ is an $A$‑bimodule, we obtain
\[
  \mathcal F(X)
  \cong
  \bigoplus_{\overline g\in \operatorname{Conj}(G)}\bigl(A\otimes_{A^e}X_{\overline g}\bigr).
\]
Recall that $A\otimes_{A^e}X$ carries the $K_\mathrm{par}G$‑action from \eqref{eq: partial representation on H0}.  Moreover, each summand 
\(
  A\otimes_{A^e}X_{\overline g}
\)
is naturally a $K_\mathrm{par}G$‑module, since $X_{\overline g}$ inherits its action from~$X$.

We then define a functor
\[
    \mathcal{F}^{\overline g}: \lmod{S^{e}}_\mathrm{gr} \rightarrow \lmod{K_\mathrm{par}G}
\]
by
\[
    \mathcal{F}^{\overline g}(X):=A\otimes_{A^e}X_{\overline g},
\]
and on morphisms $f: X\to Y$ by
\[
    \mathcal{F}^{\overline g}(f):=\operatorname{id}\otimes_{A^{e}} f|_{X_{\overline g}}.
\]

From the definition of $\mathcal{F}^{\overline{g}}$ we obtain the following proposition

\begin{proposition} \label{p: F split}
    The functor $\mathcal{F}^{\overline{g}}$ is additive, right exact and 
    \[
        \mathcal{F}|_{\lmod{S^{e}}_\mathrm{gr}}(-)= \bigoplus_{\overline{g}\in \operatorname{Conj}(G)} \mathcal{F}^{\overline{g}}(-).
    \]
\end{proposition}
\begin{corollary}\label{c: graded projectives are acyclic}
    Let \(X\) be a projective \(G\)-graded \(S\)-bimodule that is also projective as an \(S\)-bimodule.
    Then \(\mathcal{F}^{\overline{g}}(X)\) is \(H_{0}^{\mathrm{par}}(G, -)\)-acyclic for all \(\overline{g} \in \operatorname{Conj}(G)\).
\end{corollary}

\begin{corollary}\label{c: split H0 S iso}
    Let \(M\) be a \(G\)-graded \(S\)-bimodule.
    Then \( \mathbf{L}_\bullet\bigl(H^\mathrm{par}_{0}(G, - ) \circ \mathcal{F}^{\overline{g}}\bigr)(M) \cong H_\bullet^{\overline{g}}(S,M)\).
\end{corollary}

\begin{proof}
    Recall from Corollary~\ref{c: H0 F isomorphism} that the isomorphism
    $$
    \gamma:\mathcal{B}\otimes_{K_{\mathrm{par}}G}\bigl(A\otimes_{A^e}M\bigr)\to S\otimes_{S^e}M,
    $$
    defined by
    $$
    \gamma\bigl(b\otimes_{K_{\mathrm{par}}G}(a\otimes_{A^e}m)\bigr)
    =(b\cdot1_S)\otimes_{S^e}(a\,m),
    $$
    induces a natural isomorphism of functors
    $$
    H_0^{\mathrm{par}}(G,-)\,\circ\,\mathcal{F}
    \cong S\otimes_{S^e}-.
    $$
    Moreover, one checks that $\gamma$ respects the decomposition by conjugacy classes of any $G$–graded $S$–bimodule $M$, namely
    $$
    \gamma\bigl(H_0^{\mathrm{par}}(G,\,A\otimes_{A^e}M_{\overline{g}})\bigr)
    \subseteq H_0^{\overline{g}}(S,M).
    $$
    Hence, since $\gamma$ is natural, it restricts to an isomorphism
    $$
    H_0^{\mathrm{par}}(G,-)\,\circ\,\mathcal{F}^{\overline{g}}
    \overset{\gamma}{\to}
    H_0^{\overline{g}}(S,-).
    $$
    Finally, the usual spectral–sequence argument showing that $\mathrm{Tor}$ is balanced (see, e.g., \cite[Application 5.6.3]{Book_Weibel_1994_AITHA}) yields
    \[
        \mathbf{L}_\bullet\bigl(H_0^{\overline{g}}(S,-)\bigr)(M)
        \cong H_\bullet^{\overline{g}}(S,M),
    \]
    where the right-hand module is the one defined in the decomposition \eqref{eq: Hn splitting by conjugacy classes},
    for any $G$–graded $S^e$–bimodule $M$.
\end{proof}

Combining Theorem \ref{t: homology spectral sequence of an epsilon ring}, Lemma~\ref{l: EPgrad}, Proposition \ref{p: F split} and Corollary~\ref{c: split H0 S iso} we obtain the following theorem.

\begin{theorem} \label{t: spectral sequence split}
    Let $M$ be a $G$-graded $S$-bimodule. Then, 
    \[
        H_\bullet (S,M)= \bigoplus_{\overline{g} \in \operatorname{Conj}(G)} H_\bullet^{\overline{g}}(S,M),
    \]
    where $H_\bullet^{\overline{g}}(S,M)$ is the limit of a spectral sequence $_{\overline{g}}E$ such that     
    \[
        _{\overline{g}}E_{p,q}^2 = H^\mathrm{par}_p(G,H_q(A,M_{\overline{g}})) \Rightarrow  H_{p+q}^{\overline{g}}(S,M).
    \]
\end{theorem}

\subsection{Analyzing the components of \texorpdfstring{\(_{\overline{g}}E_{p,q}^{2}\)}{E}}
By \cite[Theorem 3.24]{PhDThesis_Jerez_2024}
there exists an exact functor
\[
    \Lambda: \lmod{K_\mathrm{par}G} \to \lmod{KG}
\]
such that
\[
    H^\mathrm{par}_\bullet(G, X) \cong H_\bullet(G, \Lambda(X)),
\]
for all \(K_\mathrm{par}G\)-modules \(X\).
We will use \(\Lambda\) to give a more detailed description of each component \(H_{p}^\mathrm{par}(G, H_{q}(A,M)\).

We recall from \cite[Section 2.1]{PhDThesis_Jerez_2024} that \(\Lambda(X) = (KG \otimes X)/\mathcal{K}_G(X),\)
where \(\mathcal{K}_G(X)\) is the \(K\)-submodule of \(KG \otimes X\) generated by
the set
\[
    \{ g \otimes [h] x - gh \otimes e_{h^{-1}} x : g,h \in G \text{ and } x \in X \}.
\]
We write \(\lfloor g, x \rfloor\) for the class of \(g \otimes x\) in \(\Lambda(X)\).
The left action of \(G\) on \(\Lambda(X)\) is given by \(h \cdot \lfloor g, x \rfloor = \lfloor hg, x \rfloor\).
Moreover, we have natural morphism of \(K\)-modules 
\vspace*{-15pt}\begin{center}
    \begin{minipage}[t]{0.4\textwidth}
        \begin{align*}
            \iota: X &\to \Lambda(X) \\
            x &\mapsto \lfloor 1, x \rfloor
        \end{align*}
    \end{minipage}
    \begin{minipage}[t]{0.4\textwidth}
        \begin{align*}
            \lambda: \Lambda(X) &\to X \\
            \lfloor g, x \rfloor &\mapsto [g] \cdot x
        \end{align*}
    \end{minipage}
\end{center}
such that \(\lambda \circ \iota = \operatorname{id}_{X}\) (see \cite[Section 3.2.1]{PhDThesis_Jerez_2024} for more details).

From the definition of \(\Lambda(X)\) we have the following straightforward lemma.
\begin{lemma}\label{l: well define fucntions in the globalization}
    Let \(X\) be a \(K_\mathrm{par}G\)-module and \(Y\) be a \(KG\)-module.
    Suppose that \(f: KG \otimes X \to Y\) is a \(KG\)-linear map such that
    \[
        f(g \otimes [h] x) = f(gh \otimes e_{h^{-1}} x).
    \]
    Then, \(f\) induces a well-defined \(KG\)-linear map \(\overline{f}: \Lambda(X) \to Y,\)
    given by \(\overline{f}(\lfloor g, x \rfloor) = f(g \otimes x)\).
\end{lemma}

Our goal is to prove that whenever $M$ is a $G$\nobreakdash‑graded $S$\nobreakdash‑bimodule, the module \(\Lambda\bigl(H_{\bullet}(A, M)\bigr)\)
carries a natural $G$\nobreakdash‑grading which is compatible with its induced global $G$\nobreakdash‑action.
For this, we need the following results.

\begin{lemma}\label{l: only zero has zero orbit}
    Let \(z \in \Lambda(X)\) be such that \(\lambda(g \cdot z) = 0\) for all \(g \in G\).
    Then, \(z = 0\).
\end{lemma}
\begin{proof}
    Let \(z = \sum_{i=0}^{n} \lfloor g_i, x_i \rfloor\) satisfying \(\lambda(g \cdot z) = 0\) for all \(g \in G\).
    Since \(\lambda(g_0^{-1} \cdot z) = 0\), one sees that 
    \[
        x_0 = -\sum_{i=1}^{n} [g_0^{-1}g_i] x_i.
    \]
    Hence,
    \[
        \lfloor g_0, x_0 \rfloor = -\sum_{i=1}^{n} \lfloor g_0, [g_0^{-1}g_i] x_i \rfloor = -\sum_{i=1}^{n} \lfloor g_i, e_{g_i^{-1}g_0} x_i \rfloor.
    \]
    Thus,
    \[
        z = \sum_{i=1}^{n} \lfloor g_i,x_i - e_{g_i^{-1}g_0} x_i \rfloor.
    \]
    Therefore, by induction there exists \(g \in G\) and \(x \in X\) such that \(z = \lfloor g, x \rfloor\).
    Finally, observe that
    \[
        0 = \lambda(g^{-1} \cdot z) = \lambda(\lfloor 1, x \rfloor) = x,
    \]
    whence, \(z = \lfloor g, 0 \rfloor = 0\).
\end{proof}

\begin{proposition}\label{p: G graded decomposition of Lambda X}
    Suppose that \(X\) is a left \(K_{\mathrm{par}}G\)\nobreakdash-module admitting a \(G\)\nobreakdash-grading \(X \;=\; \bigoplus_{s\in G} X_s,\)
    such that
    \([h] \cdot X_s \subseteq X_{hsh^{-1}}\) for all \(h,s\in G\).
    For each \(g\in G\), set 
    \[
        \Lambda(X)_g
        \;=\;
        \bigl\langle\,\lfloor h,x_s\rfloor 
        : hsh^{-1}=g,\;x_s\in X_s \bigr\rangle_K
        \;\subseteq\;\Lambda(X).
    \]
    Then
    \(\Lambda(X) = \bigoplus_{g\in G}\Lambda(X)_g\)
    is a \(G\)\nobreakdash-grading on \(\Lambda(X)\) such that
    \(h \cdot z \in \Lambda(X)_{hgh^{-1}}\) for all \(h,g \in G\) and \(z \in \Lambda(X)_g\).
\end{proposition}
\begin{proof}
    Clearly, \(\Lambda(X) = \sum_{g \in G} \Lambda(X)_g\) and that \(h \cdot \Lambda(X)_g \subseteq \Lambda(X)_{hgh^{-1}}\).
    It remains to show that this sum is direct.
    Let \(g \in G\) and \(z \in \Lambda(X)_g \cap \sum_{h \neq g} \Lambda(X)_h\).
    Then for every \(s \in G\),
    \(\lambda(s \cdot z) \in X_{sgs^{-1}} \cap \sum_{h \neq g} X_{shs^{-1}} = \{ 0 \}\).
    By Lemma~\ref{l: only zero has zero orbit}, it follows that \(z = 0\).
    Therefore \(\Lambda(X)_g \cap \sum_{h \neq g} \Lambda(X)_h = \{ 0 \}\).
\end{proof}

\begin{remark}\label{r: globalization grading preserve direct sums}
    Observe that the grading of \(\Lambda(X)\) preserves direct sums, i.e.,
    if \(X = \bigoplus_i X^{i}\) is a direct sum of \(G\)-graded \(K_\mathrm{par}G\)-modules
    such that \([h] \cdot x^{i} \in X^{i}_{hgh^{-1}}\)
    for all \(h,g \in G\) and \(x^{i} \in X^{i}_g\),
    then \(\Lambda(X)_g \cong \bigoplus_{i} \Lambda(X^{i})_g\).
    In particular, we can view \(X_{\overline{g}}\) itself as a \(G\)-graded \(K_\mathrm{par}G\)-module
    such that
    \[
        (X_{\overline{g}})_{t} = \left\{ 
        \begin{matrix}
            X_{t} & \text{if } t \in \overline{g}, \\
            \{ 0 \} & \text{otherwise.}
        \end{matrix}
        \right.
    \]
    Hence,
    \(X = \bigoplus_{\overline{g} \in \operatorname{Conj}(G)} X_{\overline{g}}\)
    as \(K_\mathrm{par}G\)-modules.
    Then
    \(\Lambda(X)_g \cong \bigoplus_{\overline{t} \in \operatorname{Conj}(G)} \Lambda(X_{\overline{t}})_g = \Lambda(X_{\overline{g}})_g\).
\end{remark}

To avoid confusion, if $N$ is a group and $Y$ is a $K_\mathrm{par}N$-module, we will denote its corresponding globalization functor by $\Lambda_N$. The class of $s \otimes y$ in $\Lambda_N(Y)$ will be denoted by $\lfloor s, y \rfloor_{N}$.
We denote the injection and projection maps by $\iota_N$ and $\lambda_N$, respectively.

\begin{proposition}\label{p: Psi is injective}
    Let \(G\) be a group and \(N \subseteq G\) a subgroup.
    Let \(M\) be a left \(K_{\mathrm{par}}G\)-module,
    and let \(V \subseteq M\) be a \(K\)-submodule which is also stable under the action of \(K_{\mathrm{par}}N\) (so that \(V\) becomes a left \(K_{\mathrm{par}}N\)-module).
    Then there exists an injective homomorphism of \(K_{\mathrm{par}}N\)-modules
    \(\Psi: \Lambda_{N}(V) \to \Lambda_{G}(M)\)
    such that
    \[
        \Psi(\lfloor s, x \rfloor_{N}) = \lfloor s, x \rfloor_{G}.
    \]
\end{proposition}
\begin{proof}
    By Lemma~\ref{l: well define fucntions in the globalization}, \(\Psi\) is well-defined. Indeed, let \(h', h \in N\) and \(x \in V\).
    Then
    \[
        \Psi(\lfloor h'h, e_{h^{-1}} x \rfloor_{N})
        = \lfloor h'h, e_{h^{-1}} x \rfloor_{G}
        = \lfloor h',  [h] x \rfloor_{G}
        = \Psi(\lfloor h',  [h] x \rfloor_{N}).
    \]
    Observe that \(\lambda_{N} = \lambda_G \circ \Psi\).
    Therefore, if \(z \in \ker \Psi\)
    then \(\lambda_{N}(h \cdot z) = \lambda_G(h \cdot \Psi(z)) = 0 \) for all \(h \in N\).
    Hence, by Lemma~\ref{l: only zero has zero orbit} we conclude that \(z = 0\) and therefore \(\Psi\) is injective.
\end{proof}

\begin{proposition}\label{p: char of conjugated graded modules}
    Let \(X\) be a \(K_\mathrm{par}G\)-module with a \(G\)-grading
    \(X = \oplus_{g \in G} X_g\) such that
    \([h] X_t \subseteq X_{hth^{-1}}\).
    Then, 
    \[
        H_\bullet^\mathrm{par}(G, X_{\overline{g}}) \cong H_\bullet(\mathcal{C}_g, \Lambda(X_{\overline{g}})_g),
    \]
    where \(X_{\overline{g}} = \oplus_{h \in \operatorname{Conj}(G)} X_h\) and \(\mathcal{C}_{g}\) is the centralizer of \(g\) in \(G\).
    Moreover, if for all \(t \in \overline{g}\) we have that 
    \[
        X_t = \sum_{\substack{h\in G \\ hgh^{-1}=t}} [h] \cdot X_g
    \]
    then
    \[
        H_\bullet^\mathrm{par}(G, X_{\overline{g}}) \cong H_\bullet^\mathrm{par}(\mathcal{C}_g, X_g).
    \]
\end{proposition}
\begin{proof}
    By Proposition~\ref{p: G graded decomposition of Lambda X} and Remark~\ref{r: globalization grading preserve direct sums} we obtain a \(G\)-grading
    \[
        \Lambda_G(X_{\overline{g}}) = \bigoplus_{h \in \overline{g}} \Lambda_G(X_{\overline{g}})_h,
    \]
    such that \(s \cdot \Lambda_G(X_{\overline{g}})_h = \Lambda_G(X_{\overline{g}})_{shs^{-1}}\) for all \(s \in G\) and \(h \in \overline{g}\).
    Therefore \(\Lambda_G(X_{\overline{g}})~\cong~KG~\otimes_{K \mathcal{C}_g}~\Lambda_G(X_{\overline{g}})_g\) as left \(K_\mathrm{par}G\)-modules.
    Hence, by Shapiro's lemma obtain 
    \vspace*{-6pt}
    \[
         H_\bullet(\mathcal{C}_g, \Lambda_G(X_{\overline{g}})_g)
         \cong H_{\bullet}(G, \Lambda_G(X_{\overline{g}})) 
         \overset{(\flat)}{\cong} H^\mathrm{par}_\bullet(G, X_{\overline{g}}),
    \]
    where \((\flat)\) follows from \cite[Theorem 3.24]{PhDThesis_Jerez_2024},
    which proves the first part of the proposition.

    Now we prove the second part.
    Clearly, $X_g$ is a module over $K_\mathrm{par}\mathcal{C}_g$.
    Moreover, by Proposition~\ref{p: Psi is injective},
    there exists an injective map of $\mathcal{C}_g$-modules
    \(\Psi_{g}: \Lambda_{\mathcal{C}_g}(X_g) \to \Lambda_G(X_{\overline{g}})_g \subseteq \Lambda_{G}(X_{\overline{g}})\)
    given by \(\Psi_{g}(\lfloor c, x_g \rfloor_{\mathcal{C}_g}) = \lfloor c, x_g \rfloor_G\).
    We will prove that $\Psi_{g}$ is surjective, and hence an isomorphism.
    Recall that by definition \(\Lambda_G(X_{\overline{g}})_{g}\) is generated by the set \(\{ \lfloor s, x_{t} \rfloor_{G} : sts^{-1}  = g \}\).
    Let \(h,t \in G\) such that \(hth^{-1} = g\) and \(x_t \in X_t\).
    By hypothesis, there exists families \(\{ s_i \} \subseteq G\) and \(\{ y_{g}^{i} \} \subseteq X_g\) such that \(x_t = \sum_{i} [s_i] y_g^{i}\) and \(s_{i}gs_{i}^{-1} = t\) for each \(i\).
    Notice that \(c_{i}:=hs_{i} \in \mathcal{C}_{g}\).
    Then
    \[
       \lfloor h, x_t \rfloor_G
       =
       \sum_i \lfloor h, [h^{-1}c_i] y_g^{i} \rfloor_G
       =
       \sum_i \lfloor c_i, e_{c_i^{-1}h} y_g^{i} \rfloor_G
       \in \im \Psi_g.
    \]
    Thus, \(\Psi_{g}\) is surjective, and therefore an isomorphism.
    Finally, by \cite[Theorem 3.24]{PhDThesis_Jerez_2024} we obtain
    \[
        H_{\bullet}(\mathcal{C}_{g}, \Lambda_G(X_{\overline{g}})_g)
        \overset{\Psi_g}{\cong}
        H_{\bullet}(\mathcal{C}_{g}, \Lambda_{\mathcal{C}_g}(X_g))
        \cong
        H_{\bullet}^\mathrm{par}(\mathcal{C}_g, X_g),
    \]
    which completes the proof.
\end{proof}

Combining Proposition~\ref{p: char of conjugated graded modules}
with Theorem~\ref{t: spectral sequence split}, we obtain the main theorem of this work.
\begin{theorem}\label{t: main}
    Let \(S\) be an unital epsilon-strongly \(G\)-graded algebra, set \(A := S_1\), and let \(M\) be a \(G\)-graded \(S\)-bimodule.
    Then
    \[
        H_{\bullet}(S,M) \cong \bigoplus_{\overline{g} \in \operatorname{Conj}(G)} H^{\overline{g}}_{\bullet}(S,M),
    \]
    where \(H^{\overline{g}}_{\bullet}(S,M)\) is the abutment of a homology spectral sequence \(_{\overline{g}}E\) such that
    \[
        _{\overline{g}}E_{p,q}^2 = H^\mathrm{par}_p(G, H_q(A, M_{\overline{g}})) \Rightarrow  H_{p+q}^{\overline{g}}(S,M).
    \]
    For each \(\overline{g} \in \operatorname{Conj}(G)\) the \(KG\)-module \(\Lambda(H_{q}(A,M_{\overline{g}}))\) admits a natural \(G\)-grading such that
    \[
        H^\mathrm{par}_p(G, H_q(A, M_{\overline{g}})) \cong H_{p}(\mathcal{C}_g, \Lambda(H_q(A, M_{\overline{g}}))_g).
    \]
    Finally, if in addition 
    \[
        H_{q}(A, M_t) = \sum_{\substack{h\in G \\ hgh^{-1}=t}} [h] \cdot H_{q}(A, M_g), \text{ for all } t \in \overline{g},
    \]
    then
    \[
        H^\mathrm{par}_p(G, H_q(A, M_{\overline{g}})) \cong H^\mathrm{par}_p(\mathcal{C}_g, H_q(A, M_g)), \text{ for all } p \geq 0.
    \]
\end{theorem}
\begin{proof}
    Apply Proposition~\ref{p: char of conjugated graded modules} for \(X= H_{q}(A, M_{\overline{g}})\), and make the respective substitutions in Theorem~\ref{t: spectral sequence split}.
\end{proof}

\bibliographystyle{abbrv}
\bibliography{bibliography/azu.bib}

\end{document}